\documentclass{algoritmy}

\usepackage{graphicx,amsmath,amssymb,xcolor}

\newtheorem{remark}[theorem]{Remark}

\title{Best polynomial approximation for non-autonomous linear ODEs in the $\star$-product framework\thanks{This work was supported by Charles University Research Centre program No. PRIMUS/21/SCI/009 and UNCE/24/SCI/005, and by the Magica project ANR-20-CE29-0007 funded by the French National Research Agency.}}

\author{Stefano Pozza\thanks{Faculty of Mathematics and Physics, Charles University, Sokolovská 83, 186 75 Praha 8, Czech Republic,({\tt pozza@karlin.mff.cun.cz}).}}

\begin{document}

\AlgLogo{1}{10}

\maketitle

\begin{abstract}
We present the first formulation of the optimal polynomial approximation of the solution of linear non-autonomous systems of ODEs in the framework of the so-called $\star$-product. This product is the basis of new approaches for the solution of such ODEs, both in the analytical and the numerical sense.
The paper shows how to formally state the problem and derives upper bounds for its error.
\end{abstract}

\begin{keywords} 
Non-autonomous linear ODEs, polynomial approximation, error analysis, 
\end{keywords}

\begin{AMS}
46F10, 37C60, 65L05
\end{AMS}

\pagestyle{myheadings}
\thispagestyle{plain}
\markboth{S. Pozza}{POLYNOMIAL APPROXIMATION IN THE $\star$-FRAMEWORK}

\section{Introduction}
Recently, a new approach has been introduced for the solution of systems of linear non-autonomous ordinary differential equations based on the so-called $\star$-product \cite{GiLuThJa15,GisPozInv19,Po23,Ry23}. Such an approach has proved to be valuable and effective analytically, by producing new explicit expressions for the solution of certain problems \cite{BonGis2020,GiLuThJa15,GiPo20}, and numerically, being the basis of new efficient algorithms in quantum chemistry problems \cite{BoPoVB23,PoVB24,PoVB23_PAMM_B,PoVB23_PAMM_A}. 

Given a Hermitian matrix-valued function $\tilde{A}(t) \in \mathbb{C}^{N \times N}$ analytic on the bounded interval $\mathcal{I}$, and the nonzero vector $\tilde{v}\in \mathbb{C}^{N \times N}$ we consider the initial value problem
    \begin{equation}\label{eq:ode}
        \frac{\partial}{\partial t}\tilde{u}(t) = \tilde{A}(t)\tilde{u}(t), \quad \tilde{u}(a) = \tilde{v}, \quad t \in \mathcal{I}=[a, b].
    \end{equation} 
The $\star$-product is defined over a Fréchet-Lie group on distributions \cite{Ry23}. In such a group, the initial value problem becomes a \emph{$\star$-linear system}. Thanks to this ``linearization'' of the ODE, new techniques can be applied to solve the problem. Here, we focus on the polynomial approximation approach, which can be used both in numerical approaches and in the theoretical framework. In particular, in the latter one, a symbolic algorithm named $\star$-Lanczos \cite{GiPo20} is able to produce a Krylov-like subspace approximation, that is, a polynomial approximation in the $\star$-product sense.

In this work, we will show that it is possible to formulate the problem of a best polynomial approximation for $\tilde{u}$ in the $\star$-framework. Moreover, we will show that, under some assumptions, its error can be bounded by the best uniform norm polynomial approximation error for the exponential \cite{meinardus}. This result is crucial to understand the numerical behavior of polynomial-based numerical methods when solving linear systems derived by using the $\star$-approach. Indeed, the polynomial approximation is central in the analysis of standard Krylov subspace methods (e.g., \cite{liestr}), and its extension to the $\star$-framework will also allow extending this kind of numerical analysis.

In Section~\ref{sec:basics}, we introduce the basics of the $\star$-product, and we state the main result. Section~\ref{sec:mtx} shows how to extend matrix analysis results to the $\star$-framework. The proof of the main result is given in Section~\ref{sec:app} and Section~\ref{sec:conc} draw some conclusions.

\begin{table}
\caption{List of the main $\star$-definitions and related objects, $f,g, x$ are generally from $\mathcal{A}(\mathcal{I})$.}
\begin{center} \footnotesize
\begin{tabular}{|c|c|c|c|} \hline  
Name & Symbol & Definition & Comments / Properties \\ \hline 
\lower 1ex\hbox{$\star$-product} & \lower 1ex\hbox{$f \star g$} &  \lower 1ex\hbox{$\int_\mathcal{I} f(t,\tau) g(\tau,s) \text{d}\tau $} & \lower 1ex\hbox{} \\
\lower 1ex\hbox{$\star$-identity}  &  \lower 1ex\hbox{$\delta$} & \lower 1ex\hbox{$f \star \delta = \delta \star f = f$} & \lower 1ex\hbox{$\delta(t-s)$: Dirac delta}  \\ 
\lower 1ex\hbox{$\star$-inverse} & \lower 1ex\hbox{$f^{-\star}$} & \lower 1ex\hbox{$f \star f^{-\star} = f^{-\star} \star f = \delta$} & \lower 1ex\hbox{Existence \cite{GisPozInv19,Ry23}} \\ 
\lower 1ex\hbox{Heaviside function} & \lower 1ex\hbox{$\Theta$} & \lower 1ex\hbox{$\Theta(t-s) = 1, t \geq s$, $0$ otherwise} &  \\ 
\lower 1ex\hbox{Analytic $\Theta$-set} & \lower 1ex\hbox{$\mathcal{A}_\Theta(\mathcal{I})$} & \lower 1ex\hbox{{\{$\tilde{f}(t,s)\Theta(t-s)$: \hspace{-5pt}  $\tilde{f}$ analytic on $\mathcal{I}^2$\}}} & \lower 1ex\hbox{$\star$-product-closed set}\\ 
\lower 1ex\hbox{Dirac 1st derivative} & \lower 1ex\hbox{$\delta'$} & \lower 1ex\hbox{$\delta'(t-s)$} & \lower 1ex\hbox{$\delta'\star \Theta = \Theta \star \delta' = \delta$}  \\ 
\lower 1ex\hbox{Dirac derivatives} & \lower 1ex\hbox{$\delta^{(j)}$} & \lower 1ex\hbox{$\delta^{(j)}(t-s)$} & \lower 1ex\hbox{$\delta^{(j)}\star \delta^{(i)} = \delta^{(i+j)}$}  \\ 
\lower 1ex\hbox{$\star$-powers} & \lower 1ex\hbox{$f^{\star j}$} & \lower 1ex\hbox{$f \star f \star \cdots \star f$, $j$ times} & \lower 1ex\hbox{$f^{\star 0}:= \delta$, by convention}  \\ 
\lower 1ex\hbox{$\star$-resolvent} & \lower 1ex\hbox{$R^\star(x)$} &  \lower 1ex\hbox{$\sum_{j=0}^\infty x^{\star j}, \, x \in \mathcal{A}_\Theta(\mathcal{I})$}  & \lower 1ex\hbox{$R^\star(x) = (\delta-x)^{-\star}$} \\ 
\lower 1ex\hbox{$\star$-polynomial} & \lower 1ex\hbox{$p^\star(x)$} & \lower 1ex\hbox{$\sum_{j=0}^n \alpha_j x^{\star j}$, $\alpha_j \in \mathbb{C}$} & \lower 1ex\hbox{if $\alpha_n \neq 0$, $n$ is the degree} \\ 
\hline  
\end{tabular}
\end{center} 
\label{tab:stardef} 
\end{table}  

\begin{table}
\caption{Useful properties of $\star$-product actions on $\mathcal{A}_\Theta(\mathcal{I})$ elements.}
\begin{center} \footnotesize
\begin{tabular}{|c|c|c|c|} \hline  
 Description  & Definition & Property \\ \hline 
\lower 1ex\hbox{``Integration'' in $t$} & \lower 1ex\hbox{$\tilde{F}\Theta$, $\tilde{F}(t,s)$ primitive of $\tilde{f}$  in $t$, $F(s,s)=0$}  & \lower 1ex\hbox{$\tilde{F}\Theta = \Theta \star \tilde{f}\Theta$}  \\ 
\lower 1ex\hbox{``Integration'' in $s$} & \lower 1ex\hbox{$\tilde{F}\Theta$, $\tilde{F}(t,s)$ primitive of $\tilde{f}$  in $s$, $\tilde{F}(t,t)=0$}  & \lower 1ex\hbox{$\tilde{F}\Theta = \tilde{f}\Theta \star \Theta $}  \\ 
\lower 1ex\hbox{``Differentiation'' in $t$} & \lower 1ex\hbox{$\tilde{f}^{(1,0)}\Theta$, $\tilde{f}^{(1,0)}(t,s)$ derivative of $\tilde{f}$  in $t$}  & \lower 1ex\hbox{$\delta'\star \tilde{f}\Theta = \tilde{f}^{(1,0)}\Theta + \tilde{f}\delta $}  \\ 
\lower 1ex\hbox{``Differentiation'' in $s$} & \lower 1ex\hbox{$\tilde{f}^{(0,1)}\Theta$, $\tilde{f}^{(0,1)}(t,s)$ derivative of $\tilde{f}$  in $s$}  & \lower 1ex\hbox{$\tilde{f}\Theta \star \delta' = -\tilde{f}^{(0,1)}\Theta + \tilde{f}\delta $}  \\ 
\hline  
\end{tabular}
\end{center} 
\label{tab:starprop} 
\end{table}

\section{Basics and main result}\label{sec:basics}
In order to state the main result, we first summarize the $\star$-product basics. Refer to \cite{Ry23} for the general definition of this product and the related properties.
Given a bounded interval $\mathcal{I}$, let us denote with $\mathcal{A}(\mathcal{I})$ the set of the bivariate distributions for which there exists a finite $k$ so that
\begin{equation*}
    f(t,s) = \tilde{f}_{-1}(t,s)\Theta(t-s) + \tilde{f}_{0}(t,s)\delta(t-s) + \tilde{f}_{1}(t,s)\delta'(t-s) + \dots + \tilde{f}_{k}(t,s)\delta^{(k)}(t-s),
\end{equation*}
where $\tilde{f}_{-1}, \dots, \tilde{f}_{k}$ are analytic functions over $\mathcal{I}$ both in $t$ and $s$, $\Theta$ is the Heaviside function ($\Theta(t-s) = 1$ for $t \geq s$, and $0$ otherwise), and $\delta, \delta', \dots, \delta^{(k)}$ are the Dirac delta and its derivatives.
Then, the $\star$-product of $f_1, f_2 \in \mathcal{A}(\mathcal{I})$ is
\begin{equation*}
    (f_1 \star f_2)(t,s) := \int_\mathcal{I} f_1(t,\tau) f_2(\tau,s) \,\text{d}\tau \in \mathcal{A}(\mathcal{I}).
\end{equation*}
Some of the important properties, definitions, and facts about the $\star$-product can be found in Tables~\ref{tab:stardef} and \ref{tab:starprop}.
Specifically, it is easy to see that $\delta(t-s)$ is the $\star$-product identity. Moreover, since $\mathcal{A}(\mathcal{I})$ is closed under $\star$-product, we can define the $\star$-powers of $f \in \mathcal{A}(\mathcal{I})$, denoted as $f^{\star n}$ with the convention $f^{\star 0}= \delta$ (for $f$ $\star$-invertible). Therefore, for $x \in \mathcal{A}(\mathcal{I})$, we can define the \emph{$\star$-polynomial of degree $n$} as \begin{equation}
    p^\star(t,s) := \alpha_0 \delta(t-s) + \alpha_1 x(t,s) + \alpha_2 x(t,s)^{\star 2} + \dots + \alpha_n x(t,s)^{\star n}, 
\end{equation}
with constants $\alpha_0, \dots, \alpha_n \in \mathbb{C}$, $\alpha_n \neq 0$. We call $\mathcal{P}^\star_n$ the set of all such $\star$-polynomials.

We define the subset $\mathcal{A}_\Theta(\mathcal{I}) \subset \mathcal{A}(\mathcal{I})$ of the distributions of the form $f(t,s)=\tilde{f}(t,s)\Theta(t-s)$, with $\tilde{f}$ a function analytic over $\mathcal{I}^2$. The \emph{$\star$-resolvent} is defined as
\begin{equation*}
    R^\star(x) := \sum_{j=0}^\infty x^{\star j}.    
\end{equation*}
Note that $R^\star(x)$ is well-defined (i.e., convergent) for every $x \in \mathcal{A}_\Theta(\mathcal{\mathcal{I}})$ \cite{GiLuThJa15}.

When $A, B$ are matrices with compatible sizes composed of elements from $\mathcal{A}(\mathcal{I})$, the $\star$-product straightforwardly extends to a matrix-matrix (or matrix-vector) $\star$-product. In the following, we denote with $\mathcal{A}^{N \times M}(\mathcal{I})$ and $\mathcal{A}_\Theta^{N \times M}(\mathcal{I})$ the spaces of $N \times M$ matrices with elements from those sets. We denote with $I_\star = I \delta(t-s)$ the identity matrix in $\mathcal{A}^{N \times N}(\mathcal{I})$, with $I$ the standard $N \times N$ identity matrix.

Setting $\mathcal{I} = [a,b]$, the solution $\tilde{u}(t)$ of the ODE \eqref{eq:ode} can be expressed by
\begin{equation}\label{eq:starsolres}
    \tilde{u}(t) = U(t,a)\tilde{v}, \; t\in \mathcal{I}, \quad U(t,s) = \Theta(t-s) \star R^{\star}\left(\tilde{A}(t)\Theta(t-s)\right);
\end{equation}
as proven in \cite{GiLuThJa15}.
From now on, we will skip the distribution arguments $t,s$ whenever it is useful and clear from the context.
Since $R^{\star}(\tilde{A}\Theta)$ is the $\star$-inverse of $I_\star - \tilde{A}\Theta$ (e.g., \cite{Po23}), then solving \eqref{eq:starsolres} means solving the system of $\star$-linear equations
\begin{equation*}
    (I_\star - \tilde{A}\Theta) \star x = \tilde{v}\delta, \quad \tilde{u}(t) = (\Theta \star x)(t,a) \quad t \in \mathcal{I}.
\end{equation*}
Note that this is not just a theoretical result since there is an efficient way to transform the $\star$-system into a usual linear system that can be solved numerically \cite{PoVB24,PoVB23_PAMM_B,PoVB23_PAMM_A}.

It is reasonable to consider a $\star$-polynomial approximation $p^\star(\tilde{A}\Theta)\tilde{v} \approx R^\star(\tilde{A}\Theta)\tilde{v}$.
Denoting $A(t,s) := \tilde{A}(t)\Theta(t-s)$, we specifically aim at finding the best $\star$-polynomial $p^\star(t,s)$ of degree $n$ that approximates the $\star$-resolvent $R^\star(A)\tilde{v}$ in the $L_2$ norm sense, i.e., the polynomial $q^\star$ that minimizes the error
\begin{equation*}
        \|\tilde{u}(t) - (\Theta \star q^\star(A)\tilde{v})(t,a) \|_{L_2} := \left(\int_{a}^b | \tilde{u}(\tau) - (\Theta \star q^\star(A))(\tau, a) \tilde{v} |^2 \right)^{\frac{1}{2}}, \; t \in \mathcal{I}.
\end{equation*}
Note that $\Theta \star q^\star(A) \in \mathcal{A}_\Theta^{N \times N}(\mathcal{I})$, while $q^\star(A) \in \mathcal{A}^{N \times N}(\mathcal{I})$.

\begin{theorem}[Main result]\label{thm:main}
    Consider the initial value problem \eqref{eq:ode}, with $\| \tilde{v} \|_2 = 1$.
    Let $\tilde{\lambda}_1(t), \dots, \tilde{\lambda}_N(t)$ be the eigenvalues and $\tilde{Q}(t) = [\tilde{q}_1(t), \dots, \tilde{q}_N(t)]$ the related eigenvectors of
    $\tilde{A}(t) \in \mathbb{C}^{N \times N}$.
    We define the interval 
    \begin{equation*}
        \mathcal{J} := \left[\min_{t \in \mathcal{I}, i=1,\dots, N} \tilde{\lambda}_i(t), \max_{t \in \mathcal{I}, i=1,\dots, N} \tilde{\lambda}_i(t)\right] \times \emph{length}(\mathcal{I}),
    \end{equation*}
    and denote with $E_n(\mathcal{J})$ the minimal uniform error of the polynomial approximation of the exponential over $\mathcal{J}$, i.e.,
    \begin{equation*}
        E_n(\mathcal{J}) := \min_{p \in \mathcal{P}_n} \max_{t \in \mathcal{J}} |\exp(t) - p(t)|.
    \end{equation*}
    Define $A(t,s) = \tilde{A}(t)\Theta(t-s)$ and let $\tilde{q}'_j(t)$ be the derivative of $\tilde{q}_j(t)$. If $( \tilde{q}'_j(t) \Theta(t-s) + \tilde{q}_j(t)\delta(t-s)) \star \tilde{\lambda}_j(t)\Theta(t-s) = \tilde{\lambda}_j(t)\Theta(t-s) \star (\tilde{q}'_j(t) \Theta(t-s) + \tilde{q}_j(t)\delta(t-s)),$ then the error of the $L_2$-best $\star$-polynomial approximant $q^\star$ can be bounded by
    \begin{equation*}
        \|\tilde{u}(t) - (\Theta \star q^\star(A)\tilde{v})(t,a) \|_{L_2} \leq E_n(\mathcal{J}) \, \emph{length}(\mathcal{I}) \leq  M \rho^{n+1}, \; t \in \mathcal{I}
    \end{equation*} 
    for some constant $M>0$ and $0<\rho<1$ depending on $\mathcal{J}$.
\end{theorem}

\begin{remark}
     Assume that $\mathcal{J}$ is contained (in the complex plane) in the ellipse $\mathcal{E}_{\chi}$ with foci $-1, 1$ and semi-axes whose sum is equal to the parameter $\chi$. Then, we have the explicit expression for the constants:
     \begin{equation*}
         M= \emph{length}(\mathcal{I}) \frac{2\chi}{\chi -1}\max_{z \in \mathcal{E}_\chi}|\exp(z)|, \quad \rho = \frac{1}{\chi};
     \end{equation*}
    see, e.g., \cite{BenBoi14}. As a consequence, the error in Theorem~\ref{thm:main} can be further bounded by
    \begin{equation*}
         \|\tilde{u}(t) - (\Theta \star q^\star(A)\tilde{v})(t,a) \|_{L_2} \leq \emph{length}(\mathcal{I}) \frac{2\chi}{\chi -1} \exp(\chi) \left( \frac{1}{\chi} \right)^{n+1}, \; t \in \mathcal{I}.
    \end{equation*}
    Therefore, by fixing a not-too-large value for $\chi$, the exponential convergence of the bound is not delayed too much by the bound constant. For example, if $\mathcal{J}$ is contained in $\mathcal{E}_2$, then fixing $\chi = 2$, the constant is $4 \exp(2) \,\emph{length}(\mathcal{I})$. Note that this might not be the optimal value for $\chi$. In fact, by optimizing the value of $\chi$ for every $n$, one can demonstrate a super-exponential decay of the error; see, e.g., \cite{PozSim19}. Finally, while the bound constant can grow for large intervals $\mathcal{J}$, no indefinite slow convergence is possible, i.e., $\chi$ is not constrained to be close to $1$. This is a direct consequence of the fact that the exponential is an entire function.
    \end{remark}

The proof of Theorem~\ref{thm:main} will be the outcome of the rest of the paper. The first step towards the proof is to derive an explicit form for the $\star$-monomials $f^{\star n}$ in the case in which $f(t,s)=\tilde{f}(t)\Theta(t-s) \in \mathcal{A}_\Theta(\mathcal{I})$.
\begin{lemma}\label{lemma:npower}
  Consider the function $f(t,s)= \tilde{f}(t)\Theta(t-s) \in \mathcal{A}_\Theta$ and let $\tilde{F}(t)$ be a primitive of $\tilde{f}(t)$. Then, for $ n=1, 2, \dots$, 
  \begin{align}\label{eq:npower:exp}
        f(t,s)^{\star n} &= \frac{\tilde{f}(t)}{(n-1)!} \left(\tilde{F}(t)-\tilde{F}(s)\right)^{n-1} \Theta(t-s),  \\
        \label{eq:thetanpower:exp}
        \Theta(t-s) \star f(t,s)^{\star n} &= \frac{1}{n!} \left(\tilde{F}(t)-\tilde{F}(s)\right)^{n} \Theta(t-s).
  \end{align}
  Moreover, $\Theta(t-s) \star f(t,s)^{\star 0} = \Theta(t-s)$ since $f(t,s)^{\star 0}=\delta(t-s)$ by convention.
\end{lemma}
\begin{proof}
     For $n=2$, the expression~\eqref{eq:npower:exp} is trivially obtained by
    \begin{equation*}
        f(t,s)^{\star 2} = \tilde{f}(t)\Theta(t-s) \int_s^t \tilde{f}(\tau) \,\textrm{d}\tau = \tilde{f}(t)\left(\tilde{F}(t)-\tilde{F}(s)\right)\Theta(t-s).
    \end{equation*}
    Now, by induction, assuming \eqref{eq:npower:exp} we get
    \begin{align}\label{eq:lemma:npower:1}
        f(t,s)^{\star n+1} = \frac{\tilde{f}(t)}{(n-1)!}\Theta(t-s) \int_s^t \tilde{f}(\tau)  \left(\tilde{F}(\tau)-\tilde{F}(s)\right)^{n-1} \textrm{d}\tau.
    \end{align}
    Integrating by part gives
    \begin{align*}
         \int_s^t \tilde{f}(\tau)  \left(\tilde{F}(\tau)-\tilde{F}(s)\right)^{n-1} \textrm{d}\tau &= 
          (\tilde{F}(t)- \tilde{F}(s))^{n}  - \\ 
         & (n-1) \int_s^t  \tilde{f}(\tau)\left(\tilde{F}(\tau)- \tilde{F}(s)\right)^{n-1} \,\textrm{d}\tau.
    \end{align*}
    Therefore,
     $  n \int_s^t \tilde{f}(\tau)  (\tilde{F}(\tau)-\tilde{F}(s))^{n-1} \textrm{d}\tau =  (\tilde{F}(t)- \tilde{F}(s))^{n}. $
    Together with \eqref{eq:lemma:npower:1}, this proves \eqref{eq:npower:exp}.
    Eq.~\eqref{eq:thetanpower:exp} comes from observing that
    \begin{align*}
        \Theta(t-s) \star f(t,s)^{\star n} &= \Theta(t-s) \star \frac{\tilde{f}(t)}{(n-1)!} \left(\tilde{F}(t)-\tilde{F}(s)\right)^{n-1} \Theta(t-s) \\
        &= \frac{\Theta(t-s)}{(n-1)!} \int_s^t \tilde{f}(\tau) \left(\tilde{F}(\tau)-\tilde{F}(s)\right)^{n-1} \textrm{d}\tau \\
        &= \frac{1}{n!} (\tilde{F}(t)- \tilde{F}(s))^{n} \Theta(t-s),
    \end{align*}
    which concludes the proof.
\end{proof}
See also the proof of Proposition~3.1 in \cite{Gis20}.

A consequence of Lemma~\ref{lemma:npower} is that
   $\exp(\tilde{F}(t) -\tilde{F}(s)) = \Theta(t-s) \star R^{\star}(f)(t,s),$
a well-known result; see, e.g., \cite{GiLuThJa15}.

\section{Matrix spectral decomposition and the $\star$-product}\label{sec:mtx}
Consider a time-dependent $N \times N$ Hermitian matrix-valued function $\tilde{A}(t)$ analytic over the compact interval $\mathcal{I}$. Then, for every $t \in \mathcal{I}$ there exist matrix-valued functions $\tilde{Q}(t)$ and $\tilde{\Lambda}(t)$ analytic over $\mathcal{I}$ such that:
\begin{equation}\label{eq:eigedeco}
\tilde{A}(t) = \tilde{Q}(t) \tilde{\Lambda}(t) \tilde{Q}(t)^H, \text{ with } \tilde{\Lambda}(t) = \text{diag}(\tilde{\lambda}_1(t), \dots, \tilde{\lambda}_n(t)), \; \tilde{Q}(t)^H \tilde{Q}(t) = I,    
\end{equation}
for every $t \in \mathcal{I}$; see \cite[Chapter II, Section 6]{kato} (refer to \cite{dieci99} for extensions to the non-analytic case). The elements of the diagonal matrix $\tilde{\Lambda}(t)$ are analytic functions and, for every $t\in \mathcal{I}$, the $\tilde{\lambda}_j(t)$ are the eigenvalues (eigencurves) of $\tilde{A}(t)$.  The columns of $\tilde{Q}(t)$, denoted $\tilde{q}_1(t), \dots, \tilde{q}_N(t)$, are the corresponding eigenvectors (analytic over $\mathcal{I}$).

Given $A(t,s) \in \mathcal{A}_\Theta^{N \times N}(\mathcal{I})$, the $\star$-eigenproblem is to find the $\star$-eigenvalues $\lambda(t,s)\in \mathcal{A}_\Theta(\mathcal{I})$ and the $\star$-eigenvector $q(t,s) \in \mathcal{A}^{N \times 1}(\mathcal{I})$ such that 
\begin{equation}\label{eq:stareig}
    A(t,s) \star q(t,s) = \lambda(t,s) \star q(t,s).
\end{equation}
If $\lambda(t,s)$ and $q(t,s)$ exist, then $q(t,s) \star a(t,s)$ is also a $\star$-eigenvector, for every $a(t,s)\not\equiv 0$ from $\mathcal{A}(\mathcal{I})$. 
For the specific case of interest, where $A(t,s)=\tilde{A}(t)\Theta(t-s)$, the solution to the $\star$-eigenproblem is in the following theorem.
\begin{theorem}\label{thm:stareig}
    Let $A(t,s) = \tilde{A}(t)\Theta(t-s)$ be in $\mathcal{A}_\Theta(\mathcal{I})$, and let $\tilde{\lambda}_i(t)$ and $\tilde{q}_i(t)$, be the (analytic) eigencurves and the corresponding eigenvectors as defined in \eqref{eq:eigedeco} for $i=1,\dots, N$. Then, the solution to the $\star$-eigenvalue problem \eqref{eq:stareig} is given by
    \begin{equation*}
        \lambda_i(t,s) = \tilde{\lambda}_i(t)\Theta(t-s), \quad q_i(t,s) = \tilde{q}_i'(t)\Theta(t-s) + \tilde{q}_i(t) \delta(t-s), \quad i=1,\dots, N.
    \end{equation*}
    where $\tilde{q}_i'(t)$ is the derivative of $\tilde{q}_i(t)$.
\end{theorem}
\begin{proof}
First, note that 
\begin{align*}
    \tilde{\lambda}_i(t) \delta(t-s) \star \tilde{q}_i(t) \Theta(t-s) &= \tilde{\lambda}_i(t) \int_\mathcal{I} \delta(t-\tau) \tilde{q}_i(\tau) \Theta(\tau-s) \, \text{d}\tau \\
        &= \tilde{\lambda}_i(t) \tilde{q}_i(t) \Theta(t-s) = \tilde{A}(t) \tilde{q}_i(t) \Theta(t-s) \\
        &= \tilde{A}(t) \delta(t-s) \star \tilde{q}_i(t) \Theta(t-s).
\end{align*}
Using the fact that $\tilde{\lambda}_i(t) \delta(t-s) \star \Theta(t-s) = \tilde{\lambda}_i(t) \Theta(t-s)$, and that $\delta'(t-s) \star \Theta(t-s) = \Theta(t-s) \star \delta'(t-s) = \delta(t-s)$, see Table~\ref{tab:stardef}, we obtain (we omit the variables for the sake of readability)
\begin{align*}
    \tilde{\lambda}_i \delta \star \tilde{q}_i \Theta &=  \tilde{\lambda}_i \delta \star \Theta \star \delta' \star \tilde{q}_i \Theta = \tilde{\lambda}_i \Theta \star \delta' \star \tilde{q}_i \Theta = \tilde{\lambda}_i \Theta \star q_i,
\end{align*}
where $q_i(t,s):= \delta'(t-s) \star \tilde{q}_i(t) \Theta(t-s)$. Similarly, $\tilde{A} \delta \star \tilde{q}_i \Theta = \tilde{A}\Theta \star q_i$. Combining these results, we get
\begin{align*}
    \tilde{\lambda}_i \Theta \star q_i = \tilde{\lambda}_i \delta \star \tilde{q}_i \Theta = \tilde{\lambda}_i \tilde{q}_i \Theta = \tilde{A} \tilde{q}_i \Theta = \tilde{A} \delta \star \tilde{q}_i \Theta = \tilde{A}\Theta \star q_i. 
\end{align*}
Finally, we obtain the following expression for the $\star$-eigenvectors:
\begin{align*}
    q_i(t,s) &= \delta'(t-s) \star \tilde{q}_i(t)\Theta(t-s) \\
    &= \tilde{q}_i'(t)\Theta(t-s) + \tilde{q}_i(t) \delta(t-s);
\end{align*}
see Table~\ref{tab:starprop}.
As a final remark, note that all the $\star$-products are well-defined thanks to the fact that the $\tilde{\lambda}_i(t)$ and $\tilde{q}_i(t)$ are analytic functions.
\end{proof}

Consider the matrix
\begin{equation}\label{eq:Adef}
    A(t,s) = \tilde{A}_{-1}(t,s)\Theta(t-s) + \sum_{j=0}^k \tilde{A}_j(t)\delta^{(j)}(t-s) \in \mathcal{A}^{N \times M}(\mathcal{I}),
\end{equation}
we define the Hermitian transpose of $A$ as
\begin{equation}\label{eq:AH}
    A^H(t,s) = \tilde{A}_{-1}^H(t,s)\Theta(t-s) + \sum_{j=0}^k \tilde{A}^H_j(t)\delta^{(j)}(t-s) \in \mathcal{A}^{M \times N}(\mathcal{I}),
\end{equation}
with $\tilde{A}_j^H$ the usual Hermitian transpose of a matrix.
As an immediate consequence of Theorem~\ref{thm:stareig}, we have the following $\star$-factorization of $A(t,s)$.
\begin{corollary}\label{cor:stareigdec}
    Under the same assumption of Theorem~\ref{thm:main}, we have
    \begin{equation*}
    A(t,s) = Q(t,s) \star \Lambda(t,s) \star Q(t,s)^H,
\end{equation*}
with $\Lambda(t,s) = \tilde{\Lambda}(t)\Theta(t-s)$ and $Q(t,s)= [q_1(t,s), \dots, q_N(t,s)]$. Moreover, it holds
\begin{equation*}
    Q(t,s) \star Q(t,s)^H = Q(t,s)^H \star Q(t,s) = I_\star,
\end{equation*}
that is, $Q(t,s)^H$ is the matrix $\star$-inverse of $Q(t,s)$.
\end{corollary}
\begin{proof}
    We first show that for every $i,j=1,\dots,N$ we have
    \begin{equation*}
        q_i(t,s)^H \star q_j(t,s) = \delta_{ij} \delta(t-s),
    \end{equation*}
    with $\delta_{ij}$ the Kronecker delta. Since $q_k(t,s) = \delta'(t-s) \star \tilde{q}_k \Theta(t-s)$, for $k=1,\dots,N$, then
    \begin{align*}
        q_i(t,s)^H \star q_j(t,s) &= \big(\delta'(t-s) \star \tilde{q}_i^H(t) \Theta(t-s)\big) \star \big(\delta'(t-s) \star \tilde{q}_j(t) \Theta(t-s)\big) \\
            &= \delta'(t-s) \star \big( \tilde{q}_i^H(t) \Theta(t-s) \star \delta'(t-s) \big) \star \tilde{q}_j(t) \Theta(t-s) \\
            &= \delta'(t-s) \star \big(\tilde{q}_i^H(t)\delta(t-s) \star \tilde{q}_j(t) \Theta(t-s) \big) \\
            &= \delta'(t-s) \star \tilde{q}_i^H(t)\tilde{q}_j(t)\Theta(t-s) = \delta'(t-s) \star \delta_{ij}\Theta(t-s) \\
            &= \delta_{ij}\delta(t-s).
    \end{align*}
   From Theorem~\ref{thm:stareig} we get the equality
\begin{equation*}
    A(t,s) \star Q(t,s) = Q(t,s) \star \Lambda(t,s).
\end{equation*}
The conclusion follows from $\star$-multiplying from the right by $Q(t,s)^H$. 
\end{proof}

\bigskip

Since our final goal is to measure an error, we need to introduce a $\star$-inner product and the relative $\star$-norm. To this aim, we take inspiration from the results in \cite{Ry23}, but we develop them in a different direction. 
Following \cite{Ry23}, we define the \emph{$\star$-Hermitian-transpose} of $A(t,s)$ in \eqref{eq:Adef} as
\begin{equation*}
    A^{\star H}(t,s) := A^H(s,t) = \tilde{A}^H_{-1}(s,t)\Theta(s-t) + \sum_{j=0}^k\tilde{A}_{j}^H(s)\delta^{(j)}(s-t) \in \mathcal{A}^{M \times N}(\mathcal{I}).
\end{equation*}
Roughly speaking, one has to take the usual Hermitian transpose and then swap the variable $t,s$. Note the difference with the Hermitian transpose \eqref{eq:AH}.
Now, setting $\mathcal{I} = [a,b]$, 
and given $v,w$ such that  $\Theta \star v, \Theta \star w \in \mathcal{A}_\Theta^{N\times 1}(\mathcal{I})$, for any fixed $s \in [a,b)$ we can define the inner product:
\begin{equation*}
    \langle v, w \rangle_\star(s) := \left( (\Theta \star v)^{\star H} \star \Theta \star w \right)(t,s) \Big|_{t=s} = \int_\mathcal{I} v^H(\tau,s) w(\tau,s) \, \text{d}\tau.
\end{equation*}
Note that, denoting $\Theta(t-s) \star v(t,s) = \tilde{V}(t,s)\Theta(t-s)$ and $\Theta(t-s) \star w(t,s) = \tilde{W}(t,s)\Theta(t-s)$, then
\begin{equation*}
    \langle v, w \rangle_\star(s) =  \int_s^b \tilde{V}^H(\tau,s) \tilde{W}(\tau,s) \, \text{d}\tau,
\end{equation*}
which, for the fixed $s$, is the classical inner product of the functions $\tilde{V}(\cdot,s)$ and $\tilde{W}(\cdot,s)$ on the interval $[s,b]$ (note that $v,w \equiv 0$ if and only if  $\tilde{V}, \tilde{W} \equiv 0$).
The inner product $\langle v, w \rangle_\star(s)$ is, in fact, a family of inner products depending on the parameter $s \in [a,b)$. With an abuse of notation, we refer to the function $\langle v, w \rangle_\star: [a,b) \rightarrow \mathbb{C}$ as the \emph{$\star$-inner product} of $v$ and $w$. 
Thus, again with a notation abuse, we define the \emph{$\star$-norm} as
\begin{equation*}
    \| v \|_\star(s) := \sqrt{ \langle v, v \rangle_\star }(s), \quad s \in [a,b).
\end{equation*}
The definition is justified by the following theorem.
\begin{theorem}[Properties of the $\star$-norm]\label{thm:starvecnorm}
  Given $v, w$ such that $\Theta \star v, \Theta \star w \in \mathcal{A}_\Theta^{N \times 1}(\mathcal{I})$, with $\mathcal{I}=[a,b]$, then the following properties hold for every $s \in [a,b)$:
  \begin{enumerate}
    \item $\| v \|_\star(s) \geq 0$;
    \item $\| v \|_\star(s) \equiv 0$ if and only if $v(\cdot, s) \equiv 0$;
    \item $\| \alpha v \|_\star(s) = \| \alpha \delta \star v \|_\star(s) = |\alpha| \| v \|_\star(s)$ for any scalar $\alpha \in \mathbb{C}$;
    \item $\| v + w \|_\star(s) \leq \| v \|_\star(s) + \| w \|_\star(s)$. 
\end{enumerate}
Hence $\| v \|_\star(s)$ is a norm for every fixed $s \in [a,b)$.
\end{theorem}
\begin{proof}
Since $\Theta \star v = \tilde{V}\Theta$, with $\tilde{V}$ an analytic function over $\mathcal{I}^2$, then
\begin{equation*}
    \| v \|_\star^2(s) = \int_s^b \tilde{V}^H(\tau,s) \tilde{V}(\tau,s) \, \text{d}\tau,
\end{equation*}
is the classical norm of the function $\tilde{V}(\cdot,s)$ over the interval $[s,b]$.
Thus Item~1. is trivial. Item~2. is true since $\tilde{V}(\cdot,s)\equiv 0$ over $\mathcal{I}$ if and only if $v(\cdot, s) \equiv 0$ over $\mathcal{I}$. Item~3. and 4. hold since $\Theta \star \alpha v = \alpha(\Theta \star v)$ and $\Theta \star (v + w) = \Theta \star v + \Theta \star w $.
\end{proof}

\bigskip

\begin{lemma}\label{lemma:unitary}
    Let $Q(t,s)$ as in Corollary~\ref{cor:stareigdec}, then for every $v,w \in \mathcal{A}_\Theta^{N \times 1}(\mathcal{I})$, it holds
    \begin{equation*}
        \langle Q \star v, Q \star w \rangle_\star(s) = \langle Q^H \star v, Q^H \star w \rangle_\star(s) = \langle v, w \rangle_\star(s), \quad s \in [a,b),
    \end{equation*}
    that is $Q$ and $Q^H$ are unitary with respect to $\langle \cdot, \cdot \rangle_\star$ (note that $Q^H \neq Q^{\star H}$).
\end{lemma}
\begin{proof}
Recalling that $Q(t,s) = \delta'(t-s) \star \tilde{Q}(t)\Theta(t-s)$, we get $\Theta(t-s) \star Q(t,s) = \tilde{Q}(t)\Theta(t-s)$. Therefore, 
\begin{align*}
    \langle Q \star v, Q \star w \rangle_\star(s) &= \left[(\tilde{Q}\Theta \star v)^{\star H} \star \tilde{Q}\Theta \star w\right](s,s).
\end{align*}
Now, denoting $\tilde{V}(t,s)\Theta(t-s): = \Theta(t-s) \star v(t,s)$ and $\tilde{W}(t,s)\Theta(t-s): = \Theta(t-s) \star v(t,s)$, we observe
\begin{align*}
    \tilde{Q}(t)\Theta(t-s) \star v(t,s) &= \tilde{Q}(t)\left(\Theta(t-s) \star v(t,s)\right) = \tilde{Q}(t)\tilde{V}(t,s)\Theta(t-s), \\
    \tilde{Q}(t)\Theta(t-s) \star w(t,s) &= \tilde{Q}(t)\left(\Theta(t-s) \star w(t,s)\right) = \tilde{Q}(t)\tilde{W}(t,s)\Theta(t-s).
\end{align*}
Therefore,
    \begin{align*}
    \left[(\tilde{Q}\Theta \star v)^{\star H} \star \tilde{Q}\Theta \star w\right](s,s) &= \int_s^b \tilde{V}^H(\tau,s) \tilde{Q}^H(\tau) \tilde{Q}(\tau)\tilde{W}(\tau,s) \, \text{d}\tau \\
        &= \int_s^b \tilde{V}^H(\tau,s) \tilde{W}(\tau,s) \, \text{d}\tau \\
        &= \left[(\Theta \star v)^{\star H} \star \Theta \star w \right](s,s) = \langle v, w \rangle_\star(s).
\end{align*}
Analogous arguments show that $\langle Q^H \star v, Q^H \star w \rangle_\star = \langle v, w \rangle_\star$.
\end{proof}

Finally, given a matrix $A$ so that $\Theta \star A \in \mathcal{A}_\Theta^{N \times N}(\mathcal{I})$, we define the \emph{induced matrix $\star$-norm} of $A$ as
\begin{equation}\label{eq:defmtxnorm}
    \| A \|_\star(s) := \sup_{v\not \equiv 0 \, : \, \Theta \star v \in \mathcal{A}_\Theta^{N \times 1}(\mathcal{I})} \frac{\| A \star v \|_\star(s)}{\| v \|_\star(s)}, \quad s \in [a,b).
\end{equation}
The term norm is again an abuse of notation. The following theorem explains in which sense $\| A \|_\star$ is a norm and provides several useful properties.
\begin{theorem}[Properties of the induced matrix $\star$-norm]\label{thm:starmtxnorm}
Let $A, B$ such that $\Theta \star A, \Theta \star B \in \mathcal{A}_\Theta^{N \times N}(\mathcal{I})$ with $\mathcal{I} = [a,b]$. Then, $\| A \|_\star$ satisfies the following properties for every $s \in [a,b)$:
\begin{enumerate}
    \item $\| A \|_\star(s) \geq 0$;
    \item $\| A \|_\star(s) \equiv 0$ if and only if $A(\cdot, s)\equiv 0$;
    \item $\| \alpha A \|_\star(s) = \| \alpha \delta \star A \|_\star(s) = |\alpha| \| A \|_\star(s)$ for any scalar $\alpha \in \mathbb{C}$;
    \item $\| A + B \|_\star(s) \leq \| A \|_\star(s) + \| B \|_\star(s)$;
    \item $\| A \star B \|_\star(s) \leq \| A \|_\star(s) \star \| B \|_\star(s)$.
\end{enumerate}
Therefore, $\| \cdot \|_\star(s)$ is a  sub-multiplicative matrix norm for every fixed $s \in [a,b)$.
\end{theorem}
\begin{proof}
  Items~1--4 are corollaries of Theorem~\ref{thm:starvecnorm}.
  Note that 
  \begin{equation*}
        \| A \star v \|_\star(s) = \frac{\| A \star v \|_\star(s)}{\| v \|_\star(s)} \|v \|_\star(s) \leq \| A \|_\star(s) \, \| v \|_\star(s), \quad s \in [a,b).
  \end{equation*}
  Therefore, $\| A \star B \star v \|_\star(s) \leq \| A \|_\star(s) \| B \star v \|_\star(s) \leq \| A \|_\star(s)  \| B \|_\star(s)   \| v \|_\star(s)$, proving Item~5.
\end{proof}

\section{Approximation of the $\star$-resolvent}\label{sec:app} 
Let $\tilde{v} \in \mathbb{C}^N$ be a (constant) vector and the $A(t,s)=\tilde{A}(t)\Theta(t-s) \in \mathcal{A}_\Theta^{N \times N}(\mathcal{I})$ a matrix-valued function, with $\tilde{A}(t)$ Hermitian for $t\in \mathcal{I}$ and $\mathcal{I}=[a,b]$. 
We want to find the optimal $\star$-polynomials of degree at most $n$ approximating the $\star$-resolvent $R^\star(A)\tilde{v}$ in the sense of the $\star$-norm.
That is, finding $q^\star \in \mathcal{P}_n^\star$ so that
\begin{equation}\label{eq:minprob}
    \| R^\star(A)\tilde{v} - q^\star(A)\tilde{v}  \|_\star(s) = \min_{p^\star \in \mathcal{P}^\star_n} \| R^\star(A)\tilde{v} - p^\star(A)\tilde{v}  \|_\star(s), \quad s \in [a,b).
\end{equation}
Under the assumption of Corollary~\ref{cor:stareigdec}, $A(t,s) = Q(t,s) \star \Lambda(t,s) \star Q(t,s)^H$, where the diagonal elements of $\Lambda$ are the $\star$-eigenvalues $\lambda_j(t,s) = \tilde{\lambda}_j(t)\Theta(t-s)$, $j=1,\dots,n$.
Since $Q^H \star Q = I_\star$, then we have $A^{\star 2} = Q \star \Lambda \star Q^H \star Q \star \Lambda \star Q^H = Q \star \Lambda^{\star 2} \star Q^H$, and hence
\begin{equation}\label{eq:powerdec}
    A^{\star k} = Q \star \Lambda^{\star k} \star Q^H, \quad k=0,1,2, \dots \,.
\end{equation}
\begin{theorem}\label{thm:error}
    Let $w = Q^H \star \tilde{v}\delta = Q^H\tilde{v}$. Then, the problem \eqref{eq:minprob} is equivalent to finding the $n$-degree $\star$-polynomial $q^\star$ such that
    \begin{equation*}
            \| R^\star(A)\tilde{v} - q^\star(A)\tilde{v}  \|_\star(s) = \min_{p^\star \in \mathcal{P}^\star_n} \| R^\star(\Lambda) \star w - p^\star(\Lambda) \star w  \|_\star(s), \quad s \in [a,b).
    \end{equation*}
    Moreover, for a fixed $s \in [a,b)$, the error can be bounded by
    \begin{equation*}
            \frac{\| R^\star(A)\tilde{v} - q^\star(A)\tilde{v}  \|_\star(s)}{\| \tilde{v}\delta \|_\star(s)} \leq \min_{p^\star \in \mathcal{P}^\star_n} \, \max_{i=1,\dots,N} \| R^\star(\lambda_i) - p^\star(\lambda_i)  \|_\star(s), \quad s \in [a,b),
    \end{equation*}
    where $\lambda_1(t,s), \dots \lambda_N(t,s)$ are the $\star$-eigenvalues of $A(t,s)$.
\end{theorem}
\begin{proof}
    Every $\star$-polynomial can be expanded in the $\star$-monomial basis, i.e., $p^\star(x) = \alpha_0 \delta + \alpha_1 x + \dots + \alpha_n x^{\star n}$. Therefore, by using \eqref{eq:powerdec}, we get
    \begin{equation*}
        p^\star(A) = Q \star p^\star(\Lambda) \star Q^H, \quad R^\star(A) = Q \star R^\star(\Lambda) \star Q^H. 
    \end{equation*}
    Hence, by Lemma~\ref{lemma:unitary} and Theorem~\ref{thm:starmtxnorm}, for every fixed $s \in [a,b)$, we get
    \begin{align*}
        \| R^\star(A)\tilde{v} - p^\star(A)\tilde{v}  \|_\star(s) &= \| Q \star ( R^\star(\Lambda) - p^\star(\Lambda)) \star Q^H \star \tilde{v}\delta  \|_\star(s) \\
        &= \| ( R^\star(\Lambda) - p^\star(\Lambda)) \star Q^H \star \tilde{v}\delta  \|_\star(s) \\
        &\leq \| R^\star(\Lambda) - p^\star(\Lambda) \|_\star(s) \; \| Q^H \star \tilde{v}\delta  \|_\star(s) \\
        &\leq \| R^\star(\Lambda) - p^\star(\Lambda) \|_\star(s) \; \| \tilde{v}\delta  \|_\star(s) \\
        &\leq \max_{i=1,\dots,N} \| R^\star(\lambda_i) - p^\star(\lambda_i) \|_\star(s) \; \| \tilde{v}\delta  \|_\star(s),
    \end{align*}
    concluding the proof.
\end{proof}

Now we can prove the main result of this paper. \subsection{Proof of Theorem~\ref{thm:main}}
Denote $p^\star(x) = \sum_{k=0}^n \alpha_k x^{\star k}$. By Lemma~\ref{lemma:npower} and Theorem~\ref{thm:error}, for $s \in [a,b)$ we get 
\begin{align*}
    \| R^\star(\lambda_i) - p^\star(\lambda_i)  \|_\star^2(s) &= \int_s^b \left| \sum_{k=0}^\infty \frac{(L_i(\tau, s))^k}{k!} - \sum_{k=0}^n \alpha_k\frac{(L_i(\tau, s))^k}{k!} \right|^2 \, \text{d}\tau \\
        &= \int_s^b \left| \exp(L_i(\tau, s)) - p(L_i(\tau, s)) \right|^2 \, \text{d}\tau,
\end{align*}
where $L_i(\cdot,s)$ is the primitive of $\tilde{\lambda}_i(\cdot)$ so that $L_i(s,s)=0$ and $p(t) = \sum_{k=0}^n \frac{\alpha_k}{k!}t^k$ is a (usual) polynomial. 
    Note that $L_i(\tau,s)$ is a real function with values in the interval
    \[ \mathcal{J}_i(s) = \left[\min_{\tau \in [s, b]} L_i(\tau,s), \max_{\tau \in [s, b]} L_i(\tau,s)\right]. \]
    Therefore, 
    \begin{equation*}
        \| R^\star(\lambda_i) - p^\star(\lambda_i)  \|_\star(s) \leq E_n(\mathcal{J}_i(s)) \sqrt{(b-s)}.
    \end{equation*}
        Now, by Theorem~\ref{thm:error} since $\| \tilde{v}\delta \|_\star(s) = \sqrt{b-s}$, we get
    \begin{align*}
            \left(\int_s^b | U(\tau, s)\tilde{v} - (\Theta \star q^\star(A))(\tau, s) \tilde{v} |^2\right)^{\frac{1}{2}} &= \| R^\star(A)\tilde{v} - q^\star(A)\tilde{v}  \|_\star(s) \\
             &\leq E_n(\mathcal{J}_i(s)) (b-s).
    \end{align*}
   Since the intervals $\mathcal{J}_i(s)$ are contained in $\mathcal{J}$ for every $s \in [a,b)$, 
    by the classical Bernstein's Theorem (see, e.g., \cite[page~91]{meinardus}), there exist constant $M>0$ and $0<\rho<1$ such that $E_n(\mathcal{J})(b-a) \leq M \rho^{n+1}$. Setting $s=a$ concludes the proof.

\section{Conclusion}\label{sec:conc}
The results presented are a first step in the direction of a new approach for the analysis of $\star$-product approximations of the solution of linear non-autonomous ordinary differential equations. The error analysis can affect the study of the related analytic expression and symbolic algorithms \cite{BonGis2020,GiPo20} as well as their numerical counterparts \cite{BoPoVB23} opening the way to the use of efficient Krylov subspace methods. Moreover, they can open the way to the analysis of the \emph{localization} (or \emph{decay}) phenomenon of the time-ordered exponential \cite{bensim,GiLuThJa15} by extending the techniques presented, e.g., in \cite{BenBoi14}. 
Finally, future works will try to derive more refined bounds for the error than the one in Theorem~\ref{thm:main}. This will be possible by using more information on the eigecurves' behavior.


\end{document}